\RequirePackage{iftex}\ifPDFTeX\pdfoutput=1\fi
\documentclass[11pt]{article}

\usepackage[T1]{fontenc}
\usepackage[utf8]{inputenc}
\IfFileExists{lmodern.sty}{\usepackage{lmodern}}{\usepackage{ae,aecompl}}

\usepackage[expansion=false]{microtype}
\usepackage[margin=1.15in]{geometry}
\usepackage{amsmath,amssymb,amsthm}
\usepackage{mathtools}
\usepackage{booktabs}
\usepackage{array}
\usepackage{enumitem}
\usepackage{listings}
\usepackage{xcolor}
\usepackage[colorlinks=true,linkcolor=blue!55!black,citecolor=blue!55!black,urlcolor=blue!55!black]{hyperref}
\usepackage[capitalize,nameinlink]{cleveref}

\newtheorem{theorem}{Theorem}[section]
\newtheorem{lemma}[theorem]{Lemma}

\theoremstyle{definition}

\theoremstyle{remark}

\definecolor{leankw}{RGB}{0,32,128}
\definecolor{leancm}{RGB}{96,96,96}
\definecolor{leanbg}{RGB}{248,248,246}
\lstdefinelanguage{lean}{
  keywords={theorem,lemma,def,abbrev,structure,instance,namespace,end,import,
    open,by,fun,match,with,if,then,else,let,intro,exact,have,show,set_option,in,
    where,deriving,do,return,forall,exists},
  sensitive=true,
  comment=[l]{--},
  morecomment=[s]{/-}{-/},
  string=[b]",
}
\newcommand{\code}[1]{\texttt{#1}}

\newcommand{\Powerful}{\code{Powerful}}

\title{\bfseries A Kernel-Certified Verification of the\\
Erd\H{o}s--Mollin--Walsh Conjecture below $10^{14}$}

\author{Ibrahim Mian \qquad Shayaan Siddique\\[2pt]
  \normalsize Millennium Research\\
  \normalsize\texttt{\{ibby,shayaan\}@millenniumresearch.ai}\\
  \normalsize\texttt{ibrahimnmian@gmail.com}, \texttt{shayaansiddique02@gmail.com}}

\date{July 2026}

\begin{document}
\maketitle

\begin{abstract}
Erd\H{o}s problem 364 asks whether three consecutive \emph{powerful}
numbers exist, where $n$ is powerful if $p \mid n$ implies $p^2 \mid n$.
Erd\H{o}s (1976) and, independently, Mollin and Walsh (1986) conjectured
that none do; the $abc$ conjecture implies at most finitely many. The
conjecture remains open. We present the first verification of the
conjecture at any finite bound that is checked end to end by a proof
kernel: machine-checked theorems in Lean~4 establishing that no triple of
consecutive powerful numbers exists below $10^{12}$ and below $10^{14}$,
with the axiom footprint of both theorems being exactly
$\{\code{propext}, \code{Classical.choice}, \code{Quot.sound}\}$ --- no
\code{sorry}, no \code{native\_decide}, no trusted external computation.
The statements are phrased in the byte-identical vocabulary of the
\code{google-deepmind/formal-conjectures} formalization of the problem,
and we prove abstractly that the open conjecture implies each bounded
form, pinning the statement correspondence. The proof reduces the search
to odd numbers via a mod-$4$ argument, represents every odd powerful
number as $a^2 b^3$ with $a,b$ odd and $b$ squarefree, enumerates all odd
powerful numbers by a fueled, kernel-reducible generator whose
completeness is proved once and instantiated across $3{,}524$ per-interval
Boolean certificates, and eliminates the seven surviving distance-$2$
pairs (the members of OEIS A076445 below $10^{14}$) by explicit
non-powerfulness witnesses. Every certificate's expected values are
computed independently by a Python engine, so each kernel-checked equality
doubles as a cross-implementation agreement. Total certified kernel time
is roughly $46$ CPU-hours. Larger \emph{uncertified} computations exist
(exhaustive to $10^{22}$; conditionally to about $7.38 \times 10^{28}$);
our contribution is not a computational record but the elimination of
trusted enumeration code from the evidence chain.
\end{abstract}

% ===========================================================================
\section{Introduction}\label{sec:intro}

A positive integer $n$ is \emph{powerful} (also \emph{squarefull}) if
every prime dividing $n$ divides it at least twice: $p \mid n$ implies
$p^2 \mid n$. Powerful numbers are exactly the integers of the form
$a^2 b^3$ \cite{Golomb1970,OEISA001694}. Pairs of consecutive powerful numbers exist
in abundance --- $(8,9)$ is the smallest, and Pell-type constructions give
infinitely many --- but no triple $(n, n{+}1, n{+}2)$ of consecutive
powerful numbers has ever been found. Erd\H{o}s conjectured that none
exists \cite{Erdos1976}; Mollin and Walsh independently made the same
conjecture a decade later \cite{MollinWalsh1986}. The problem appears as
B16 in Guy's \emph{Unsolved Problems in Number Theory} \cite{GuyB16} and
as problem 364 in Bloom's Erd\H{o}s problem catalogue
\cite{ErdosProblems364}. It is a well-known consequence of the $abc$
conjecture that at most finitely many such triples exist; unconditionally
the question is wide open, with recent partial results ruling out triples
of restricted shapes \cite{Beckon2019,Chan2025,She2025}.

Computational evidence for the conjecture is substantial but rests
entirely on \emph{trusted code}. Donovan Johnson's 2011 enumeration of
consecutive powerful pairs, recorded in the b-file of OEIS A060355
\cite{OEISA060355}, is exhaustive below $10^{22}$ but uncertified and
methodologically undocumented. The thirteen known pairs of powerful
numbers differing by exactly $2$ (OEIS A076445
\cite{OEISA076445}) reach about $7.38 \times 10^{28}$, but the search is
exhaustive only if that list is complete, which has not been established;
indeed the OEIS entry was renamed between 2013 and 2014 to drop the word
``consecutive,'' retracting the implication that no powerful number lies
strictly between the members of a listed pair. Alekseyev's conjectured
33-term extension (2012) reaches about $8.1 \times 10^{66}$ and is
explicitly not known to be exhaustive. In every tier, the reader must
trust an unpublished program, its compiler, and its operator.

This paper reports a different kind of evidence. We have produced
machine-checked proofs, in the Lean~4 proof assistant \cite{Lean4} over
the \code{mathlib} library \cite{mathlib}, of the following two theorems,
where \Powerful{} is carried byte-identically from the
\code{google-deepmind/formal-conjectures} repository's formal statement of
problem 364 \cite{FormalConjectures}:

\begin{lstlisting}
theorem Erdos364.no_powerful_triple_up_to_1e12 :
    ∀ n : ℕ, n + 2 ≤ 10^12 →
      ¬ (Powerful n ∧ Powerful (n+1) ∧ Powerful (n+2))

theorem Erdos364.no_powerful_triple_up_to_1e14 :
    ∀ n : ℕ, n + 2 ≤ 10^14 →
      ¬ (Powerful n ∧ Powerful (n+1) ∧ Powerful (n+2))
\end{lstlisting}

Both compile in Lean 4.30.0 against \code{mathlib} v4.30.0, and
\code{\#print axioms} reports for each exactly
\[
  \{\code{propext},\ \code{Classical.choice},\ \code{Quot.sound}\},
\]
the three axioms of Lean's standard classical foundation. In particular
there is no \code{sorry} (no admitted subgoal), no \code{native\_decide}
(no reliance on the compiled evaluator, which would extend the trusted
base to the Lean compiler and C toolchain), and no custom axiom. Every
arithmetic fact consumed by the proof --- the enumeration of powerful
numbers, integer square roots, squarefreeness tests, merges, scans, and
tilings --- is evaluated by the Lean kernel itself via \code{decide},
reducing closed Boolean terms to \code{true}.

\paragraph{What is and is not claimed.}
We state the claim's two tiers plainly, as the accompanying repository
does. This is a \emph{certification} result, not a computational record:
the uncertified computations above reach much farther. What the present
work adds is that below $10^{14}$, every step of the argument is checked
by a proof kernel rather than trusted contributor code, and the certified
statement is syntactically the restriction of the community's canonical
formalization of the open conjecture. To our knowledge this is the first
kernel-certified verification of the Erd\H{o}s--Mollin--Walsh conjecture
at any finite bound.

\paragraph{Contributions.}
\begin{enumerate}[leftmargin=2em]
\item The two theorems above, with axiom records, build logs, and
  reproduction scripts committed alongside the sources
  (\cref{sec:repro}).
\item A proof architecture (\cref{sec:architecture}) that factors a
  large kernel computation into a once-proved parametric soundness
  theorem plus thousands of small Boolean certificates, keeping
  per-process kernel memory bounded ($\le 7.9$\,GB) where a monolithic
  \code{decide} fails below $10^9$.
\item A \emph{reference-table} optimization (\cref{sec:table}) that
  amortizes the dominant per-certificate cost (squarefreeness retesting)
  into a single kernel-verified table, yielding a measured $23.8\times$
  end-to-end speedup and making the $10^{14}$ rung affordable
  ($\approx 30$ CPU-hours).
\item A methodology of \emph{dual-implementation certificates}
  (\cref{sec:trust}): every expected value baked into a certificate is
  computed independently by a Python engine, so each kernel-checked
  equality is simultaneously a cross-implementation agreement; plus an
  axiom gate run in continuous integration and a freeze discipline
  binding the attested axiom records to the exact sources.
\item A statement-level bridge (\cref{sec:bridge}) proving abstractly
  that the upstream open conjecture implies each bounded theorem, so the
  certified results are exactly finite fragments of the conjecture, in
  the same words.
\end{enumerate}

All sources, certificates, logs, and documentation are public under the
Apache~2.0 license at
\begin{center}\url{https://github.com/ibrahimmian36/Optio}\end{center}

% ===========================================================================
\section{Background and related work}\label{sec:related}

\subsection{Powerful numbers and problem 364}

Golomb \cite{Golomb1970} initiated the systematic study of powerful
numbers and observed the representation $n = a^2 b^3$ with $b$
squarefree, unique in that normal form. Sentance \cite{Sentance1981}
studied occurrences of consecutive \emph{odd} powerful numbers, i.e.\
pairs differing by $2$, the objects our scan isolates. Mollin and Walsh
\cite{MollinWalsh1986} developed the theory of powerful pairs via
Pell equations and posed the triple conjecture independently of Erd\H{o}s
\cite{Erdos1976}. Recent work constrains hypothetical triples: Beckon
\cite{Beckon2019} gives a mod-$36$ constraint; Chan \cite{Chan2025} and
She \cite{She2025} rule out triples of special shapes near cubes. None of
these approaches yields a finite verification bound; they are
complementary to the present work.

\subsection{Machine-checked large computations}

Landmark formal developments have folded heavy computation into
kernel-checked proofs, notably the Four-Color Theorem in Coq
\cite{Gonthier2008} and the Kepler conjecture in HOL Light and Isabelle
\cite{Hales2017}. Within Lean, large \code{decide}-based certificates are
increasingly common, but two engineering constraints shape any such
effort: kernel reduction does not unfold well-founded recursion (so
library functions like \code{Nat.sqrt} are unusable inside
\code{decide}), and reduction of a single large closed term can exhaust
memory, since the kernel retains the whole term graph. Our architecture
--- structural fuel recursion in accumulator style, and many small
certificates checked in separate processes, composed by a parametric
soundness theorem --- is a direct response to both constraints, and we
document the measured wall (a monolithic check succeeds in $12.6$\,s at
$10^8$ but thrashes past $1.7$\,GB of term graph at $10^9$) that forced
it.

\subsection{Formalized conjecture statements}

The \code{google-deepmind/formal-conjectures} repository
\cite{FormalConjectures} maintains formal statements of open problems,
including Erd\H{o}s 364. Formal \emph{statements} of open problems are
only useful anchors if downstream results speak their exact vocabulary;
otherwise a translation layer must itself be trusted. We therefore carry
the upstream \Powerful{} definition byte-identically (up to removal of
module-system markers our file layout does not need), pin the upstream
commit (\code{e923379e6}, 2026-07-21) in the source header, and prove the
implication from the upstream conjecture to each bounded theorem
(\cref{sec:bridge}).

% ===========================================================================
\section{Formal setting}\label{sec:setting}

\subsection{The statement vocabulary}

The upstream definition, reproduced in our \code{Erdos364/Defs.lean}, is:

\begin{lstlisting}
namespace Nat

/-- `n` is `k`-full: every prime dividing `n` does so at least `k` times. -/
def Full (k : ℕ) (n : ℕ) : Prop := ∀ p ∈ n.primeFactors, p^k ∣ n

/-- A powerful number: every prime in `n` appears squared. -/
abbrev Powerful : ℕ → Prop := (2).Full

end Nat
\end{lstlisting}

Two conventions are inherited with the definition: $0$ and $1$ are both
\Powerful{} (their \code{primeFactors} sets are empty). This is harmless
for the triple statement, since $2$ is not powerful, so no triple can
begin at $0$ or $1$ anyway. A \code{Decidable} instance for
\code{Nat.Full} follows by unfolding, which is what allows the witness
lemmas of \cref{sec:witness} to discharge divisibility side conditions by
\code{decide}.

\subsection{Trust model}

Lean's kernel accepts a proof term only if it type-checks against the
axioms in scope. The \code{decide} tactic produces a proof of a decidable
proposition $P$ by exhibiting the evaluation of its decision procedure to
\code{true}; the variant \code{decide +kernel} forces the reduction to be
performed by the kernel itself. We forbid \code{native\_decide}
throughout: it delegates evaluation to compiled native code, which is
dramatically faster but enlarges the trusted base from the kernel to the
whole compiler pipeline. Our axiom gate (\cref{sec:trust}) mechanically
enforces the absence of \code{native\_decide} (no \code{\_native}
constants), the absence of \code{sorryAx}, and the containment of every
theorem's axioms in $\{\code{propext}, \code{Classical.choice},
\code{Quot.sound}\}$.

% ===========================================================================
\section{Proof architecture}\label{sec:architecture}

The proof composes six mathematical steps, each proved once and
parametrically, with the per-interval computations isolated into Boolean
certificates. \Cref{fig:pipeline} summarizes the reduction. Throughout,
$X$ denotes the verification bound ($10^{12}$ or $10^{14}$).

\begin{figure}[t]
\centering
\begin{tabular}{c}
triple $(n,n{+}1,n{+}2)$ of powerful numbers, $n+2 \le X$ \\[2pt]
$\Downarrow$ {\small mod-$4$ reduction (\cref{sec:mod4})} \\[2pt]
$n$ odd; $(n, n{+}2)$ an \emph{odd powerful pair} at distance $2$ \\[2pt]
$\Downarrow$ {\small tiling (\cref{sec:tiling})} \\[2pt]
the pair lands in the window $[lo_i, hi_i{+}2]$ of a unique chunk \\[2pt]
$\Downarrow$ {\small enumeration completeness $+$ adjacency (\cref{sec:generator,sec:sorted})} \\[2pt]
$n$ appears in that chunk's kernel-verified expected-pair list \\[2pt]
$\Downarrow$ {\small confinement check (\cref{sec:assembly})} \\[2pt]
$n \in \{25,\, 70225,\, 130576327,\, 189750625,\, 512706121225,$ \\
$\phantom{n \in \{} 13837575261123,\, 99612037019889\}$ \\[2pt]
$\Downarrow$ {\small witness kills (\cref{sec:witness})} \\[2pt]
$\lnot\,\Powerful(n{+}1)$ --- contradiction
\end{tabular}
\caption{The reduction chain. Double arrows are once-proved lemmas;
the boxed middle steps consume the $320$ (resp.\ $3{,}204$)
per-chunk Boolean certificates.}
\label{fig:pipeline}
\end{figure}

\subsection{Step 1: the mod-4 reduction}\label{sec:mod4}

\begin{lemma}[\code{not\_powerful\_of\_two\_mod\_four}]
If $n \equiv 2 \pmod 4$ then $n$ is not powerful.
\end{lemma}

\begin{lemma}[\code{odd\_of\_powerful\_triple}]\label{lem:odd}
If $n$, $n+1$, $n+2$ are all powerful then $n$ is odd.
\end{lemma}

\begin{proof}[Proof sketch]
A number $\equiv 2 \pmod 4$ has $2$ among its prime factors while $4$
does not divide it, refuting $2$-fullness. If $n$ were even, then one of
$n$, $n+2$ is $\equiv 2 \pmod 4$.
\end{proof}

Consequently a triple forces $(n, n{+}2)$ to be a pair of \emph{odd}
powerful numbers at distance exactly $2$ --- precisely the objects
enumerated by OEIS A076445 --- and the search space halves before any
computation begins.

\subsection{Step 2: the representation lemma}\label{sec:representation}

The one lemma with genuine mathematical content underwrites the
completeness of the enumeration:

\begin{lemma}[\code{exists\_odd\_sq\_mul\_cube}]\label{lem:rep}
Every odd powerful $m$ admits a representation $m = a^2 b^3$ with $a$,
$b$ odd and $b$ squarefree.
\end{lemma}

\begin{proof}[Proof sketch]
\code{mathlib}'s square-times-squarefree decomposition writes
$m = s^2 f$ with $f$ squarefree. For any prime $p \mid f$, the exponent
of $p$ in $m$ is $2 v_p(s) + 1$; powerfulness makes this at least $2$,
hence $v_p(s) \ge 1$, so $f \mid s$ (comparing factorizations).
Substituting $s = f c$ gives $m = c^2 f^3$. Oddness of $c$ and $f$ is
inherited from $m$ since divisors of odd numbers are odd. The proof is
classical bookkeeping over \code{mathlib}'s \code{Nat.factorization}
API. Uniqueness is deliberately not needed: duplicate entries in the
enumeration are harmless, because merging keeps duplicates adjacent and
the scan logic tolerates them.
\end{proof}

\subsection{Step 3: a fueled, kernel-reducible enumerator}\label{sec:generator}

The enumerator must run \emph{inside} the kernel, which imposes a strict
discipline: no well-founded recursion (it does not reduce), so every loop
is structural recursion on an explicit fuel argument, in accumulator
style to bound pending-term depth. The primitives are:

\begin{itemize}[leftmargin=2em]
\item \code{isqrt}: integer square root by binary search with constant
  fuel $64$, sufficient for all $n < 2^{64}$. Its correctness
  ($r^2 \le n < (r{+}1)^2$) is proved by induction on fuel
  (\code{isqrtAux\_correct}), together with the Galois-style
  characterization \code{le\_isqrt\_iff}: $a \le \code{isqrt}\ n
  \leftrightarrow a^2 \le n$.
\item \code{sqfreeAux}: trial-square squarefreeness, exact when run with
  fuel \code{isqrt b}, per lemma\\ \code{sqfreeAux\_isqrt\_iff}.
\item \code{genOddRangeAux}: the ascending stream $(2(k_0{+}i){+}1)^2
  \cdot b^3$ for $i < \mathit{cnt}$, i.e.\ the odd values of $a$ with
  $a^2 b^3$ in the target window.
\item \code{outerRangeAux}: one stream per odd squarefree $b$ with
  $b^3 \le hi$, where the per-$b$ window of $a$ is computed by two
  \code{isqrt} calls; \code{mergeAux}/\code{mergeAll}: fueled balanced
  merging of the streams into one ascending list.
\end{itemize}

\begin{theorem}[chunk-level completeness]\label{thm:complete}
If $m$ is odd and powerful with $lo \le m \le hi$, and the fuel
parameter $kb$ satisfies $hi < (2kb{+}1)^3$, then $m$ is a member of
\code{oddPowerfulRange lo hi kb}.
\end{theorem}

\begin{proof}[Proof sketch]
By \cref{lem:rep}, $m = a^2 b^3$ with $a, b$ odd and $b$ squarefree.
From $b^3 \le m \le hi < (2kb{+}1)^3$ the generator's outer loop visits
$b$, and the exact squarefreeness test admits it; the \code{isqrt}
characterization places $a$ inside the generated window. Membership
survives merging (\code{mem\_mergeAll}).
\end{proof}

The certificate for each chunk supplies the hypotheses of
\cref{thm:complete} as \emph{literal} side conditions --- Boolean
inequalities among the chunk's numeric parameters --- checked by a single
\code{decide} over the chunk table (\cref{sec:assembly}).

\subsection{Step 4: sortedness, adjacency, and the scan}\label{sec:sorted}

The pair-finding argument is where sortedness pays off:

\begin{lemma}[adjacency]\label{lem:adj}
In a sorted list all of whose members are odd, if $m$ and $m+2$ are both
members, then some \emph{adjacent} pair of entries equals $(m, m{+}2)$.
\end{lemma}

\begin{proof}[Proof sketch]
The only value that could separate $m$ from $m+2$ in a sorted list is
$m+1$, which is even and hence absent. Duplicates of $m$ or $m+2$ only
shift which adjacent pair witnesses the claim.
\end{proof}

The linear scan \code{scanGap2Aux} collects exactly the adjacent pairs
at distance $2$, per lemma \code{scanGap2Aux\_catches}. Composing
\cref{thm:complete,lem:adj} with the sortedness and parity invariants of
the merge (\code{sorted\_mergeAll}, \code{stream\_all\_odd}, and
companions) yields
the once-proved parametric soundness theorem that every chunk certificate
instantiates for free:

\begin{lstlisting}
theorem checkChunk_sound {lo hi kb cnt : Nat} {exp : List Nat}
    (hlo : 1 ≤ lo) (hhi : hi + 2 < 2 ^ 64)
    (hkb : hi + 2 < (2 * kb + 1) * (2 * kb + 1) * (2 * kb + 1))
    (hkb40 : kb ≤ 2 ^ 40)
    (hcheck : checkChunk lo hi kb cnt exp = true) :
    ∀ m : Nat, Odd m → m.Powerful → (m + 2).Powerful →
      lo ≤ m → m + 2 ≤ hi + 2 → m ∈ exp
\end{lstlisting}

That is: a single \code{true} Boolean --- comparing the enumeration's
length and its gap-$2$ scan against expected literals --- certifies that
\emph{every} odd powerful pair in the chunk's window appears in the
expected list. The direction matters: the certificate needs only
completeness of the capture (no pair escapes), not soundness of every
listed entry, since listed entries are subsequently killed anyway.

\subsection{Step 5: tiling and stitching}\label{sec:tiling}

The chunks $[lo_i, hi_i]$ must tile $[1, X]$ exactly. A single Boolean
fold \code{tiles} checks $lo_1 = 1$, $lo_{i+1} = hi_i + 1$, and
$hi_{\mathrm{last}} = X$ over the literal boundary table; the covering
lemma \code{mem\_of\_tiles} then places every $1 \le m \le X$ in some
chunk. Because each chunk's scan window extends to $hi_i + 2$ (overlap
$2$), a pair $(m, m{+}2)$ near a boundary is caught by the chunk owning
$m$: no pair escapes at a seam.

\subsection{Step 6: witness kills}\label{sec:witness}

Below $10^{14}$, the union of all expected-pair lists is exactly the
first seven terms of A076445. Each pair's middle is refuted by an
explicit witness prime dividing it exactly once:

\begin{lemma}[\code{not\_powerful\_of\_witness}]
If $p$ is prime, $p \mid n$, and $p^2 \nmid n$, then $n$ is not
powerful.
\end{lemma}

\begin{table}[h]
\centering
\begin{tabular}{r r c}
\toprule
pair opener $n$ & middle $n+1$ & witness $p$ \\
\midrule
$25$ & $26$ & $2$ \\
$70{,}225$ & $70{,}226$ & $2$ \\
$130{,}576{,}327$ & $130{,}576{,}328$ & $29$ \\
$189{,}750{,}625$ & $189{,}750{,}626$ & $2$ \\
$512{,}706{,}121{,}225$ & $512{,}706{,}121{,}226$ & $2$ \\
$13{,}837{,}575{,}261{,}123$ & $13{,}837{,}575{,}261{,}124$ & $19$ \\
$99{,}612{,}037{,}019{,}889$ & $99{,}612{,}037{,}019{,}890$ & $2$ \\
\bottomrule
\end{tabular}
\caption{The seven odd powerful pairs below $10^{14}$ (A076445
$a(1)$--$a(7)$) and the witness prime refuting powerfulness of each
middle. Pairs opening $1 \bmod 4$ have even middles divisible by $2$
exactly once; the two pairs opening $3 \bmod 4$ need odd witnesses.}
\label{tab:witnesses}
\end{table}

Divisibility and primality side conditions are discharged by
\code{decide} and \code{norm\_num}; each kill is a two-line instance.

\subsection{Step 7: assembly}\label{sec:assembly}

The assembly theorem for the $10^{12}$ rung is stated \emph{conditionally}
on the conjunction of the chunk certificates,

\begin{lstlisting}
theorem no_powerful_triple_up_to_1e12_of
    (hall : ∀ e ∈ C12.table, ChunkSpec.check e = true) :
    ∀ n : ℕ, n + 2 ≤ 1000000000000 →
      ¬ (Nat.Powerful n ∧ Nat.Powerful (n + 1) ∧ Nat.Powerful (n + 2))
\end{lstlisting}

and its proof is exactly the chain of \cref{fig:pipeline}: mod-$4$
opener, tiling placement, literal side conditions by one table-level
\code{decide}, \code{checkChunk\_sound}, a table-level \code{decide}
confining every expected pair to the known list, and the witness kills.
Everything in the assembly file compiles in seconds without the
certificates. The unconditional theorem then lives in a separate module
importing the $320$ chunk certificates and discharging \code{hall};
building it re-checks every chunk in the kernel.

% ===========================================================================
\section{The reference-table optimization for $10^{14}$}\label{sec:table}

\subsection{Cost attribution}

Scaling from $10^{12}$ to $10^{14}$ naively would multiply cost past
affordability, so an attribution study preceded any optimization.
Per-chunk kernel cost was measured to track \emph{value magnitude}, not
entry count: chunk $0$ of $320$ costs $3.8$\,s while chunk $319$ costs
$242$--$305$\,s at essentially the same $\sim$2{,}500 entries. Five
instrumented kernel runs on a fixed chunk (baseline; length-only;
generation-only; generation with the squarefree guard removed; isolated
\code{isqrt} folds), each checked against an independently computed
Python mirror value, attributed the dominant cost to the per-chunk
squarefreeness retesting of all candidate cube bases $b$ and the per-$b$
binary searches on large literals --- work that is \emph{identical across
chunks} of a rung.

\subsection{The kernel-verified base table}

The optimization reifies that shared work once per rung.
\code{mkBTable}~$kb_T$ constructs the ascending list of odd squarefree
$b \le 2 kb_T - 1$; for $X = 10^{14}$, $kb_T = 23{,}208$. Soundness needs
exactly two facts, both proved: \emph{completeness} --- every odd
squarefree $b$ in range is a member (\code{mem\_mkBTable}; the
catastrophic direction if lost, since a missing $b$ silently drops part
of the enumeration) --- and \emph{parity} --- every member is odd
(\code{mkBTable\_all\_odd}; feeding the adjacency argument).
Squarefreeness of members is a cost concern rather than a soundness need
and is deliberately not proved: spurious members would only add entries,
which the expected-value equality would expose. The table itself is
committed as a literal (\code{bTable1e14}) and proved equal to
\code{mkBTable 23208} by a single one-time kernel evaluation; the
table-driven checker \code{checkChunkT} then walks the literal list
instead of retesting.

Trust in the optimized path was earned by \emph{revalidation}: all $320$
chunks of the already-certified $10^{12}$ rung were regenerated as
\code{checkChunkT} certificates with byte-identical boundaries, counts,
and expected lists (mechanically diffed against the certified table
before any kernel time was spent), and all $320$ passed --- $2{,}422$
kernel-seconds versus the certified $57{,}652$, a $23.8\times$ speedup,
with the certified path untouched.

\subsection{The $10^{14}$ certificate set}

The $10^{14}$ rung comprises $3{,}204$ table-driven certificates with
variable-width windows (narrower at the top end to bound memory), for
example:

\begin{lstlisting}
theorem chunk_0000 :
    Erdos364.Spike.checkChunkT 1 10310522 2500
      Erdos364.Spike.bTable1e14 [25, 70225] = true := by
  decide +kernel
\end{lstlisting}

The assembly (\code{no\_powerful\_triple\_up\_to\_1e14\_of}) is
conditional on exactly two kernel facts, both discharged in the final
module: the rung-table equality \code{bTable1e14 = mkBTable 23208}, and
the conjunction of all $3{,}204$ chunk certificates. The table-driven
soundness theorem \code{checkChunkT\_sound} takes the table's
completeness and parity facts as hypotheses, which the assembly derives
from \code{mem\_mkBTable} and \code{mkBTable\_all\_odd} through the
proved table equality.

% ===========================================================================
\section{Engineering, cost, and reproduction}\label{sec:cost}

\begin{table}[t]
\centering
\begin{tabular}{l r r}
\toprule
& $X = 10^{12}$ & $X = 10^{14}$ \\
\midrule
checker & \code{checkChunk} (per-chunk $kb$) & \code{checkChunkT} (shared table) \\
chunk certificates & $320$ & $3{,}204$ \\
kernel time (certificates) & $57{,}652$\,s ($\approx 16$ CPU-h) & $107{,}603$\,s ($\approx 30$ CPU-h) \\
max single chunk & $\approx 305$\,s & $61$\,s \\
peak per-process memory & $7.9$\,GB & $6.9$\,GB \\
pairs found & $5$ & $7$ \\
\bottomrule
\end{tabular}
\caption{Measured certified-computation cost per rung (committed logs in
\code{data/chunk\_runs/}). The one-time $10^{14}$ table-equality check
adds minutes of kernel work and requires lifting the elaborator's
default heartbeat ceiling.}
\label{tab:cost}
\end{table}

\Cref{tab:cost} summarizes the measured cost. The chunked design exists
because a monolithic \code{decide} hits a kernel memory wall between
$10^8$ ($12.6$\,s) and $10^9$ (term graph past $1.7$\,GB resident,
thrashing without completion); one Boolean equality per interval, each in
its own process, resets kernel memory per certificate. Certificates were
batch-checked by resumable drivers with bounded parallelism; the
$10^{14}$ build wants tens of GB free, and a committed pod script
(\code{scripts/pod\_final14.sh}) reproduces it on a 64\,GB machine.

\subsection{Reproduction}\label{sec:repro}

The soundness library --- every lemma through the conditional headline
theorems --- builds on an ordinary machine:

\begin{lstlisting}
lake exe cache get && lake build && scripts/axiom_gate.sh
\end{lstlisting}

The unconditional theorems require re-checking the certificate modules
($\approx 46$ CPU-hours total) on a large-memory machine. Their
\code{\#print axioms} outputs are committed
to the repository, in \code{data/chunk\_runs/}, with full build and
batch logs alongside.

\subsection{Run ledger and freeze discipline}

The development's run ledger (\code{docs/\allowbreak PHASE2\_LOG.md}) records every
failure encountered en route, including two reporting bugs found and
fixed in our own harness: a doc-comment/\allowbreak\code{set\_option} parse trap
that detached an option from its theorem, a default-heartbeat abort of
the table-equality check, and a driver script whose log filtering masked
that failure as success. After certification, the Lean sources of the
certified closure were \emph{frozen} as attested: the committed axiom
records attest builds of those exact sources, so cosmetic lint cleanups
(two known docstring slips, neither affecting any statement or proof) are
deliberately not applied, and all subsequent work (gate additions, audit,
bridge, CI) lands only in new files outside the certified closure.

% ===========================================================================
\section{The trust story}\label{sec:trust}

Three mutually reinforcing mechanisms carry the result's credibility.

\paragraph{Kernel checking with a minimal axiom footprint.}
Both headline theorems depend on exactly \code{propext},
\code{Classical.choice}, and \code{Quot.sound}. No \code{sorry}, no
\code{native\_decide}, no ad hoc axiom: the trusted base is Lean's kernel
and its standard classical axioms.

\paragraph{Dual implementation.}
Every literal a certificate compares against --- entry counts and
gap-$2$ pair lists per chunk, and the mirror values of the attribution
study --- is computed \emph{independently} by a Python engine
(type-checked under \code{mypy --strict}, linted, unit-tested). A
certificate can only evaluate to \code{true} if the kernel-reduced Lean
enumeration agrees with the Python enumeration on that chunk; the
$3{,}524$ green certificates therefore constitute $3{,}524$
cross-implementation agreements in addition to their role in the proof.

\paragraph{A mechanical gate in CI.}
The axiom gate checks a curated $61$-theorem manifest
(\code{\#print axioms} on the headline theorems and every named lemma of
the architecture), a mechanical $260$-theorem whole-library audit (so
nothing slips the manifest), and the committed certificate records;
continuous integration runs it on every push. The gate's scope is stated
honestly: it re-verifies the lemma library that builds within CI budgets,
while the certificate modules are attested by their committed pod-build
axiom records, since $3{,}524$ kernel-checked certificates exceed free CI
allowances.

% ===========================================================================
\section{The bridge to the open conjecture}\label{sec:bridge}

The upstream formalization states the conjecture as
\begin{lstlisting}
theorem erdos_364 :
    ¬ ∃ (n : ℕ), Powerful n ∧ Powerful (n + 1) ∧ Powerful (n + 2)
\end{lstlisting}
Because our \Powerful{} is byte-identical to upstream's, no translation
layer needs to be trusted; the relationship between the open statement
and our bounded theorems is itself proved, abstractly and in both useful
directions:

\begin{lstlisting}
/-- The upstream conjecture (were it proved) implies each bounded form. -/
theorem bounded_of_erdos364
    (h : ¬ ∃ n : ℕ, Nat.Powerful n ∧ Nat.Powerful (n + 1) ∧
      Nat.Powerful (n + 2)) (X : ℕ) :
    ∀ n : ℕ, n + 2 ≤ X →
      ¬ (Nat.Powerful n ∧ Nat.Powerful (n + 1) ∧ Nat.Powerful (n + 2))

/-- A witness below any bound would refute the conjecture outright. -/
theorem erdos364_false_of_witness {X n : ℕ} (_ : n + 2 ≤ X)
    (hn : Nat.Powerful n ∧ Nat.Powerful (n + 1) ∧ Nat.Powerful (n + 2)) :
    ∃ m : ℕ, Nat.Powerful m ∧ Nat.Powerful (m + 1) ∧ Nat.Powerful (m + 2)
\end{lstlisting}

The bounded theorems are thus exactly finite fragments of
\code{erdos\_364}: the same claim, restricted, in the same words. The
contrapositive shape is the direction a hypothetical counterexample
discovered by any future search would travel.

% ===========================================================================
\section{Limitations and comparison}\label{sec:limits}

We restate the two tiers. \emph{Certified:} no triple below $10^{14}$,
kernel-checked end to end. \emph{Uncertified, larger:} Johnson's
enumeration is exhaustive below $10^{22}$ but rests on undocumented code
\cite{OEISA060355}; the A076445 list reaches $\approx 7.38 \times
10^{28}$ but is exhaustive only if complete, which is not established
\cite{OEISA076445}; Alekseyev's extension reaches $\approx 8.1 \times
10^{66}$ and is explicitly not known to be exhaustive. Notably, none of
the $33$ known pairs across all tiers has a powerful middle, consistent
with the conjecture.

The certified bound is limited by kernel-evaluation cost, which our
measurements show scales primarily with value magnitude. The natural next
levers --- baked per-$b$ skip hints with a tolerant checker, narrower
top-end windows, and parallel fan-out --- were scoped during the
optimization pass; a $10^{16}$ rung appears plausible at roughly an order
of magnitude more kernel time, but we certify only what the measured cost
curve honestly supports. A complementary direction is certifying a
\emph{pair} enumeration (A060355-style, allowing even members) at lower
bounds, and contributing the bounded statements upstream to
\code{formal-conjectures}.

% ===========================================================================
\section{Conclusion}\label{sec:conclusion}

Finite verification cannot settle the Erd\H{o}s--Mollin--Walsh
conjecture. What it can do is be \emph{trustworthy}, and the standard for
computational evidence in number theory can rise from ``a program its
author trusts found nothing'' to ``a proof kernel with a three-axiom
footprint checked every step.'' Below $10^{14}$, the conjecture now meets
that standard: the enumeration of powerful numbers, the reduction to odd
pairs, the capture of every pair at distance $2$, and the refutation of
every candidate middle are theorems, with the statement pinned
byte-for-byte to the community's formalization of the open problem. The
architecture --- one parametric soundness theorem, thousands of small
dual-implementation certificates, a mechanical axiom gate, and a freeze
discipline binding evidence to sources --- is portable to other bounded
verifications of open conjectures, which we hope it will see.

\paragraph{Acknowledgements.}
The kernel-discipline enumeration technique adapts prior in-house
work on fueled kernel computation within the proof kernel. We thank the maintainers of \code{mathlib}
and of the \code{formal-conjectures} repository, as well as
T.~F.~Bloom for the Erd\H{o}s problem catalogue. Development was carried
out with the assistance of Claude (Anthropic).

% ===========================================================================

\appendix

% ===========================================================================
\section{Certified module inventory}\label{app:modules}

The certified closure comprises the following Lean modules (frozen as
attested by the committed axiom records). Line counts are for the
soundness library; the certificate sets are generated.

\begin{center}\small
\begin{tabular}{l l}
\toprule
module & role \\
\midrule
\code{Defs} & the upstream \Powerful{} vocabulary, byte-identical \\
\code{Spike} & fueled kernel primitives and the chunk checker \\
\code{Mod4} & step 1: the mod-$4$ opener reduction \\
\code{Representation} & step 2: odd powerful $= a^2 b^3$, $b$ squarefree \\
\code{Generator} & step 3: \code{isqrt}, squarefreeness, stream completeness \\
\code{Sorted} & step 4: merge invariants, adjacency, \code{checkChunk\_sound} \\
\code{Tiling} & step 5: the Boolean tiling fold and covering lemma \\
\code{Witness} & step 6: the seven middle kills (\cref{tab:witnesses}) \\
\code{Cert}, \code{C12/*} & $10^{12}$ chunk specs and $320$ certificates \\
\code{BTable}, \code{TableGen} & the base table and \code{checkChunkT\_sound} \\
\code{BTableData1e14}, \code{C14/*} & the $10^{14}$ table and $3{,}204$ certificates \\
\code{Assembly}, \code{Assembly14} & step 7: conditional headline theorems \\
\code{Main}, \code{Main14} & the unconditional theorems and axiom prints \\
\code{Bridge} & the statement-level bridge of \cref{sec:bridge} \\
\bottomrule
\end{tabular}
\end{center}

Instrumentation modules retained for the record ---
\code{BTableData1e8}, \code{BTableData1e12}, and \code{CostAttrib} ---
are not part of the certified closure.


\begin{thebibliography}{99}\small

\bibitem{Erdos1976}
P.~Erd\H{o}s,
\emph{Problems and results on number theoretic properties of consecutive
integers and related questions},
Proceedings of the Fifth Manitoba Conference on Numerical Mathematics
(1976), 25--44.

\bibitem{MollinWalsh1986}
R.~A.~Mollin and P.~G.~Walsh,
\emph{On powerful numbers},
International Journal of Mathematics and Mathematical Sciences
\textbf{9} (1986), 801--806.

\bibitem{Golomb1970}
S.~W.~Golomb,
\emph{Powerful numbers},
American Mathematical Monthly \textbf{77} (1970), 848--852.

\bibitem{Sentance1981}
W.~A.~Sentance,
\emph{Occurrences of consecutive odd powerful numbers},
American Mathematical Monthly \textbf{88} (1981), 272--274.

\bibitem{GuyB16}
R.~K.~Guy,
\emph{Unsolved Problems in Number Theory},
3rd ed., Springer, 2004, Problem B16.

\bibitem{Beckon2019}
E.~Beckon,
\emph{On consecutive triples of powerful numbers},
Rose-Hulman Undergraduate Mathematics Journal \textbf{20}(2) (2019).

\bibitem{Chan2025}
T.~H.~Chan,
\emph{A note on three consecutive powerful numbers},
Integers \textbf{25} (2025), \#A7. arXiv:2503.21485.

\bibitem{She2025}
J.~She,
\emph{Nonexistence of consecutive powerful triplets around cubes with
prime-square factors},
Integers \textbf{25} (2025), \#A103. arXiv:2507.16828.

\bibitem{ErdosProblems364}
T.~F.~Bloom,
\emph{Erd\H{o}s problems: problem 364},
\url{https://www.erdosproblems.com/364}.

\bibitem{OEISA001694}
OEIS Foundation,
\emph{A001694: Powerful numbers},
The On-Line Encyclopedia of Integer Sequences,
\url{https://oeis.org/A001694}.

\bibitem{OEISA060355}
OEIS Foundation,
\emph{A060355: Numbers $n$ such that $n$ and $n+1$ are powerful},
The On-Line Encyclopedia of Integer Sequences,
\url{https://oeis.org/A060355}. B-file enumeration by D.~Johnson (2011).

\bibitem{OEISA076445}
OEIS Foundation,
\emph{A076445: Pairs of powerful numbers differing by 2},
The On-Line Encyclopedia of Integer Sequences,
\url{https://oeis.org/A076445}. Terms by J.~McCranie (2002),
J.~Reynolds (2005), T.~D.~Noe (2006); conjectured extension by
M.~Alekseyev (2012).

\bibitem{Lean4}
L.~de~Moura and S.~Ullrich,
\emph{The Lean 4 theorem prover and programming language},
in Automated Deduction -- CADE 28, LNCS 12699, Springer, 2021, 625--635.

\bibitem{mathlib}
The mathlib Community,
\emph{The Lean mathematical library},
in Proceedings of CPP 2020, ACM, 2020, 367--381.

\bibitem{FormalConjectures}
Google DeepMind,
\emph{formal-conjectures: a repository of formalized open problems},
\url{https://github.com/google-deepmind/formal-conjectures},
file \code{FormalConjectures/ErdosProblems/364.lean},
commit \code{e923379e609b9d5987011a1d1f06ec22ea25cd20} (2026-07-21).

\bibitem{Gonthier2008}
G.~Gonthier,
\emph{Formal proof --- the Four-Color Theorem},
Notices of the American Mathematical Society \textbf{55} (2008),
1382--1393.

\bibitem{Hales2017}
T.~Hales et al.,
\emph{A formal proof of the Kepler conjecture},
Forum of Mathematics, Pi \textbf{5} (2017), e2.

\end{thebibliography}
\end{document}